\documentclass[preprint]{imsart}
\pdfoutput=1
\RequirePackage[OT1]{fontenc}
\usepackage{amsthm,amsmath,natbib,amssymb,bm,enumerate}
\RequirePackage[colorlinks,citecolor=blue,urlcolor=blue]{hyperref}
\usepackage{booktabs}

\startlocaldefs
\numberwithin{equation}{section}
\theoremstyle{plain}
\newtheorem{thm}{Theorem}[section]
\newtheorem{lemma}{Lemma}[section]

\theoremstyle{remark}
\newtheorem{remark}{Remark}[section]

\endlocaldefs
\def\citeapos#1{\citeauthor{#1}'s (\citeyear{#1})}

\def\DS{{ \mathrm{\scriptscriptstyle DS} }}
\def\JS{{ \mathrm{\scriptscriptstyle JS} }}
\def\ZH{{ \mathrm{\scriptscriptstyle ZH} }}
\def\LP{{ \mathrm{\scriptscriptstyle LP} }}

\def\PR{{\mathrm{Pr}}}

\allowdisplaybreaks

\begin{document}

\begin{frontmatter}
\title{$\ell_p$-norm based James-Stein estimation with minimaxity and sparsity}
\runtitle{James-Stein estimation with sparsity}

\begin{aug}
\author{\fnms{Yuzo} \snm{Maruyama}
\ead[label=e1]{maruyama@csis.u-tokyo.ac.jp}}
\runauthor{Y. Maruyama}

\affiliation{The University of Tokyo}
\address{University of Tokyo \\
\printead{e1}}
\end{aug}
\begin{abstract}
A new class of minimax Stein-type shrinkage estimators of a multivariate normal mean is studied
where the shrinkage factor is based on an $\ell_p$ norm. The proposed estimators allow some but not
all coordinates to be estimated by $0$ thereby allow sparsity as well as minimaxity.
\end{abstract}


\end{frontmatter}

\section{Introduction}
Let $Z\sim N_d(\theta,\sigma^2I_d)$.
We are interested in estimation of the mean vector $\theta$
with respect to the quadratic loss function 
$ L(\delta,\theta)=\sum\nolimits_{i=1}^d(\delta_{i}-\theta_{i})^2/\sigma^2$.
Obviously the risk of $z$ is $d$. 
We shall say one estimator is as good as the other if the former has a risk no greater than 
the latter for every $\theta$. 
Moreover, one dominates the other if it is as good as the other and has smaller risk for
some $\theta$. In this case, 
the latter is called inadmissible.
Note that $z$ is a minimax estimator, that is, 
it minimizes $\sup_\theta E[L(\delta,\theta)]$ among all estimators $\delta$.
Consequently any $\delta$ is as good as $z$ if and only if it is minimax.

\cite{Stein-1956} showed that $z$ is inadmissible when $d\geq 3$. 
\cite{James-Stein-1961} explicitly found a class of minimax estimators 
\begin{align*}
  \hat{\theta}_{\JS}=\left(1-\frac{c\sigma^2}{\|z\|_2^2}\right)z
\end{align*}
with $0\leq c\leq 2(d-2)$ and $\|z\|^2_2=\sum_{i=1}^d z_i^2$.
\cite{Baranchik-1964} proposed the James-Stein positive-part estimator
\begin{equation}
  \hat{\theta}^+_{\JS}=\max\left(0,1-\frac{c\sigma^2}{\|z\|_2^2}\right)z
\end{equation}
with $0< c\leq 2(d-2)$ which dominates the James-Stein estimator.
A problem with the James-Stein positive-part estimator is, however, that it selects
only between two models: the origin and the full model. 
\cite{Zhou-Hwang-2005} overcome the difficulty by utilizing the so-called $\ell_p$-norm
given by
\begin{equation}\label{def:pnorm}
\textstyle \|z\|_p=\left\{\sum\nolimits_{i=1}^d|z_i|^p\right\}^{1/p}
\end{equation}
and in fact proposed minimax estimators $ \hat{\theta}^+_{\ZH}$ with the $i$-th component given by
\begin{equation}
 \hat{\theta}^+_{i\ZH} = \max\left( 0, 1-\frac{c\sigma^2}{\|z\|^{2-\alpha}_{2-\alpha}|z_i|^{\alpha}}\right)z_i
\end{equation}
where $ 0\leq \alpha <(d-2)/(d-1)$ and $ 0<c\leq 2\left\{(d-2)-\alpha(d-1)\right\} $.
When $\alpha>0$, the $i$-th component of the estimator with
\begin{equation}\label{eq:sparse_z_i}
 \frac{|z_i|}{\sigma} \leq c^{1/\alpha}\left(\frac{\sigma}{\|z\|_{2-\alpha}}\right)^{(2-\alpha)/\alpha}
\end{equation}
becomes zero.
Hence the choice between a full model and
reduced models, where some coefficients are reduced to zero, is possible.

In this paper, we establish minimaxity of a new class of $\ell_p$-norm based shrinkage estimators 
$\hat{\theta}^+_{\LP}$ with the $i$-th component given by
\begin{equation}\label{eq:shrinkage_estimator_1}
 \hat{\theta}^+_{i\LP} 
= \max\left(0,1-\frac{c\sigma^2}{\|z\|_p^{2-\alpha}|z_i|^\alpha}\right)z_{i} 
\end{equation}
where $0\leq \alpha<(d-2)/(d-1)$, $p>0$,  $0< c\leq 2(d-2)\gamma(d,p,\alpha)$ and
\begin{equation*}
 \gamma(d,p,\alpha)=\min(1,d^{(2-p-\alpha)/p})\left\{1-\alpha\frac{d-1}{d-2}\right\}.
\end{equation*}
When $\alpha$ is strictly positive in \eqref{eq:shrinkage_estimator_1}, sparsity happens as in \eqref{eq:sparse_z_i}.
In \cite{Zhou-Hwang-2005}, $p=2-\alpha$ was assumed and the $\ell_p$-norm with 
\begin{equation*}
d/(d-1)<p[=2-\alpha]<2
\end{equation*}
was treated. From their proof, the choice of $p=2-\alpha$ 
seemed only applicable for constructing estimators with minimaxity and sparsity simultaneously.
We produce such minimax estimators based on the $\ell_p$-norm for all $p>0$.
As an extreme  case ($p=\infty$), we can show that
\begin{equation*}
 \max\left(0,1-\sigma^2\frac{2\{(d-2)-\alpha(d-1)\}}{d \left\{\max |z_i|\right\}^{2-\alpha}|z_i|^{\alpha}}
\right)z_i
\end{equation*}
with $0\leq \alpha<(d-2)/(d-1)$ is minimax.
A more general result of minimaxity, 
corresponding to  the result of \cite{Efron-Morris-1976},
where $c$ is replaced by $\phi(\|z\|_p/\sigma)$ in \eqref{eq:shrinkage_estimator_1},
is given in Section \ref{sec:main_1}. 
In Section \ref{sec:main_2}, the corresponding results for unknown $\sigma^2$ are presented.

\section{Minimaxity with sparsity: known scale}
\label{sec:main_1}
In this section, we assume that $\sigma^2$ is known and
establish conditions under which estimators $\hat{\theta}_\phi$ of the form
\begin{equation}\label{eq:shrinkage_estimator_3}
 \hat{\theta}_{i\phi} 
= \left(1-\frac{\sigma^2\phi(\|z\|_p/\sigma)}{\|z\|_p^{2-\alpha}|z_i|^\alpha}\right)z_{i}
\end{equation}
as the $i$-th component, are minimax.
Note the shrinkage factor of \eqref{eq:shrinkage_estimator_3},
$ 1-\sigma^2\phi(\|z\|_p/\sigma)/\left\{\|z\|_p^{2-\alpha}|z_i|^\alpha\right\}$ 
is symmetric with respect to $z_i$.
As shown in Theorem 4 of \cite{Zhou-Hwang-2005}, 
the shrinkage estimator with the symmetry is dominated by the positive-part estimator.
Hence the minimaxity of $\hat{\theta}^+_\phi$ follows from the minimaxity of $\hat{\theta}_\phi$.

Recall that the risk of $z$ is equal to $d$ or finite.
Hence a straightforward application of Schwarz's inequality shows that the risk of
$z+\xi(z)$ is finite if and only if
\begin{equation}\label{eq:risk_finite}
 E\left[\sum\nolimits_{i=1}^d\{\xi_i(z)\}^2/\sigma^2\right]<\infty.
\end{equation}
In that case, \citeapos{Stein-1981} identity states
that if $\xi(z)$ is absolutely continuous, we have
\begin{equation}\label{Stein_Identity}
 E\left[(z_i-\theta_i)\xi(z)\right]=\sigma^2 E\left[(\partial/\partial z_i)\xi(z)\right]
\end{equation}
for $i=1,\dots,d$ and each expectation exists.

In this paper, we assume $0\leq \alpha< 1$ and $\phi$ is bounded,
say $|\phi|\leq M$ for some $M>0$.
Under these assumptions, \eqref{eq:risk_finite} follows with $\xi=(\xi_1,\dots,\xi_d)'$
and
\begin{equation}\label{eq:gammagamma}
 \xi_i(z)=\hat{\theta}_{i\phi}-z_i=-\frac{\sigma^2\phi(\|z\|_p/\sigma)}{\|z\|_p^{2-\alpha}|z_i|^\alpha}z_i.
\end{equation}
In fact, we have
\begin{equation*}
 \frac{\sum_{i=1}^d\{\xi_i(z)\}^2}{\sigma^2}=
  \sigma^2\phi(\|z\|_p/\sigma)^2\frac{\|z\|_{2(1-\alpha)}^{2(1-\alpha)}}{\|z\|_p^{2(2-\alpha)}}
\leq \sigma^2M^2\frac{\|z\|_{2(1-\alpha)}^{2(1-\alpha)}}{\|z\|_p^{2(2-\alpha)}}
\end{equation*}
and further 
\begin{equation*} 
 \frac{\|z\|_{2(1-\alpha)}^{2(1-\alpha)}}{\|z\|_p^{2(2-\alpha)}}
    \leq \frac{\max(1,d^{(p-2+2\alpha)/\{2p(1-\alpha)\}})}{\|z\|_{2(1-\alpha)}^{2}}\leq \frac{\max(1,d^{(p-2+2\alpha)/\{2p(1-\alpha)\}})}{\|z\|_{2}^{2}}
\end{equation*}
by Part \ref{lem:norm:1} of Lemma \ref{lem:norm} in Appendix.
Since $E[\sigma^2/\|z\|_{2}^{2}]\leq 1/(d-2)$, for $\xi$ given by \eqref{eq:gammagamma} we have
\begin{equation*}
 E\left[\sum\nolimits_{i=1}^d\{\xi_i(z)\}^2/\sigma^2\right]\leq \frac{M^2\max(1,d^{(p-2+2\alpha)/\{2p(1-\alpha)\}})}{d-2}.
\end{equation*}
Hence under the assumption of bounded $\phi$, the risk of $\hat{\theta}_\phi$ given by
\eqref{eq:shrinkage_estimator_3} is finite.
Further, with an additional assumption that $\phi$ is absolutely continuous,
\citeapos{Stein-1981} identity given by \eqref{Stein_Identity} is available 
for derivation of \citeapos{Stein-1981} unbiased risk estimator.
%
\begin{lemma}\label{lem:stein_1}
Assume that $\phi(v)$ is bounded and absolutely continuous and that $ 0\leq \alpha < 1$. 
\begin{enumerate}
 \item \label{lem:stein_1_1}
The risk function of the
estimator $\hat{\theta}_\phi$ is
\begin{align} \label{eq:risk}
 E\left[ \frac{\|\hat{\theta}_\phi-\theta\|_2^2}{\sigma^2}\right]  
 =d+E\left[\sum\nolimits_{i}\left(\frac{|z_i|}{\|z\|_p}\right)^{p-\alpha}
 \frac{\phi(\|z\|_p/\sigma)\psi_\phi(z/\sigma)}{(\|z\|_p/\sigma)^2}
\right]
\end{align}
where
 \begin{equation}\label{eq:psi_phi}
\begin{split}
 \psi_\phi(z)&=
 \frac{\phi(\|z\|_p)}{\|z\|_{p}^{-p-\alpha+2}}
 \frac{\sum_{i}|z_i|^{2(1-\alpha)}}{\sum\nolimits_{i}|z_i|^{p-\alpha}}
-2(1-\alpha)\|z\|_p^p \frac{\sum_i |z_i|^{-\alpha}}{\sum\nolimits_{i}|z_i|^{p-\alpha}}\\
&\quad -2 \left\{\alpha-2+\|z\|_p\frac{\phi'(\|z\|_p)}{\phi(\|z\|_p)}\right\}.
\end{split}
 \end{equation}
\item \label{lem:stein_1_2}
Assume $\phi(v)\geq 0$. Then $ \psi_\phi(z) \leq \Psi_\phi(\|z\|_p)$ where
\begin{equation}\label{eq:Psi_phi}
\Psi_\phi(v) =
\max(1,d^{(p+\alpha-2)/p})\phi(v)-2\left\{d-2-\alpha(d-1)\right\}-2\frac{v\phi'(v)}{\phi(v)}.
\end{equation}
\end{enumerate}
\end{lemma}
\begin{proof}
From the invariance with respect to the transformation,
 $z\to cz$, we can take $c=1/\sigma$ and hence, without the loss of geniality, assume
 $\sigma^2=1$ in the proof.

[Part \ref{lem:stein_1_1}] Let $v=\|z\|_p$. 
Componentwisely we have
\begin{equation}\label{eq:quadratic_2}
\begin{split}
& (\hat{\theta}_{i}-\theta_{i})^2 
=\left\{\left(1-\phi(v)v^{\alpha-2}|z_i|^{-\alpha}\right)z_{i}-\theta_{i}\right\}^2 \\
&=(z_{i}-\theta_{i})^2+
\phi^2(v)v^{2(\alpha-2)}|z_{i}|^{2(1-\alpha)} -2(z_{i}-\theta_{i})
\left\{ \phi(v)v^{\alpha-2}|z_i|^{-\alpha}z_i\right\}.
\end{split} 
\end{equation}
For the third term of the right-hand side of \eqref{eq:quadratic_2}, 
the Stein identity given by \eqref{Stein_Identity} is applicable.
Note
\begin{equation}\label{eq:dif_p_norm}
 \frac{\partial}{\partial z_{i}}v=v^{1-p}|z_i|^{p-2}z_i, \  
\frac{\partial}{\partial z_{i}}\left\{|z_i|^{-\alpha}z_i\right\}=(1-\alpha)|z_i|^{-\alpha}.
\end{equation}
Then the differentiation of $\phi(v)v^{\alpha-2}|z_i|^{-\alpha}z_i$ with respect to
$z_i$ is given by
\begin{align*}
&(1-\alpha)\phi(v)v^{\alpha-2}|z_i|^{-\alpha} 
 +(\alpha-2)\phi(v)v^{\alpha-p-2}|z_i|^{p-\alpha}
 +\phi'(v)v^{\alpha-p-1}|z_i|^{p-\alpha} \\
&=\left\{\phi(v)v^{\alpha-p-2}\right\}\left\{(1-\alpha) v^p |z_i|^{-\alpha} 
+\left\{(\alpha-2)+v\phi'(v)/\phi(v)\right\}|z_i|^{p-\alpha}\right\}
\end{align*}
and Part \ref{lem:stein_1_1} follows by taking summation with respect to $i$.

[Part \ref{lem:stein_1_2}]  Recall $0\leq \alpha < 1$ and $p>0$.
By Part \ref{lem:norm:0} of Lemma \ref{lem:norm} in Appendix, we have
\begin{equation}\label{eq:cor_ineq}
\sum_{i=1}^d |z_{i}|^{-\alpha} \geq d\frac{\sum_{i=1}^d |z_{i}|^{p-\alpha}}
{\sum_{i=1}^d |z_{i}|^p}= d\frac{\sum_{i=1}^d |z_{i}|^{p-\alpha}}{\|z\|_p^p}
\end{equation}
and, by Part \ref{lem:norm:2} of Lemma \ref{lem:norm}, 
\begin{align}\label{eq:cor_ineq_111}
\frac{1}{\|z\|_{p}^{-p-\alpha+2}}\frac{\sum_{i}|z_i|^{2(1-\alpha)}}{\sum\nolimits_{i}|z_i|^{p-\alpha}}
=\frac{\sum_{i}s_i^{2(1-\alpha)/p}}{\sum\nolimits_{i}s_i^{(p-\alpha)/p}}
 \leq \max(1,d^{(p+\alpha-2)/p})
\end{align}
where $s_i=|z_i|^p/\|z\|_p^p$ with $\sum_{i=1}^ds_i=1$ and $s_i\geq 0$ for any $i$.
By applying these inequalities to \eqref{eq:psi_phi}, Part \ref{lem:stein_1_2} follows.
\end{proof}

By Lemma \ref{lem:stein_1}, 
a sufficient condition for $E[\|\hat{\theta}-\theta\|_2^2]\leq d$ is
\begin{equation}
\Psi_\phi(v)\leq 0
\end{equation}
as well as the assumption of Lemma \ref{lem:stein_1}.
When $\phi$ is monotone non-decreasing,
we easily have a following result for minimaxity,
which corresponds to the result by \cite{Baranchik-1970} with $\alpha=0$ and $p=2$.
\begin{thm}
\label{thm:minimax_0}
Assume $d\geq 3$ and $0\leq \alpha <(d-2)/(d-1)$. 
Assume $\phi(v)$ is absolutely continuous, monotone non-decreasing and 
\begin{equation*}
 0\leq \phi(v)\leq 2(d-2)\gamma(d,p,\alpha)
\end{equation*}
 where $\gamma(d,p,\alpha)$ is given by
\begin{equation}\label{eq:gamma}
 \gamma(d,p,\alpha)=\min(1,d^{(2-p-\alpha)/p})\left\{1-\alpha\frac{d-1}{d-2}\right\}.
\end{equation}
 Under known $\sigma^2$, the shrinkage estimator $\hat{\theta}_\phi$, with the $i$-th component,
\begin{equation*}
 \hat{\theta}_{i\phi} 
= \left(1-\frac{\sigma^2\phi(\|z\|_p/\sigma)}{\|z\|_p^{2-\alpha}|z_i|^\alpha}\right)z_{i}
\end{equation*}
is minimax.
\end{thm}
More generally, by the derivative,
\begin{equation}\label{eq:derivative}
\begin{split}
& \frac{d}{dv}\left\{\frac{v^b\phi(v)}{\{a-\phi(v)\}^c}\right\} \\ &=
\frac{bv^{b-1}\phi(v)}{\{a-\phi(v)\}^{c+1}}\left(\frac{c-1}{b}v\phi'(v)+\frac{a}{b}\frac{v\phi'(v)}{\phi(v)}+a-\phi(v)\right),
\end{split} 
\end{equation}
we have a following sufficient condition as in \cite{Efron-Morris-1976}.
\begin{thm}
\label{thm:minimax}
Assume $d\geq 3$ and $0\leq \alpha <(d-2)/(d-1)$. 
Assume $\phi(v)$ is absolutely continuous and 
\begin{equation*}
 0\leq \phi(v)\leq 2(d-2)\gamma(d,p,\alpha).
\end{equation*}
Further, for all $v$ with $ \phi(v)< 2(d-2)\gamma(d,p,\alpha)$
\begin{equation*}
g_\phi(v)= \frac{v^{d-2-\alpha(d-1)}\phi(v)}
{2(d-2)\gamma(d,p,\alpha)-\phi(v)}
\end{equation*}
is assumed to be non-decreasing. 
Further if there exists $v_*>0$ 
such that $ \phi(v)= 2(d-2)\gamma(d,p,\alpha)$,
then $ \phi(v)$ is assumed equal to 
$2(d-2)\gamma(d,p,\alpha)$ for all $v\geq v_*$.
Then $\hat{\theta}_\phi$ is minimax.
\end{thm}
Recall that $\ell_p$ norm with any positive $p$ is available in Lemma \ref{lem:stein_1}
and Theorem \ref{thm:minimax}.
As an extreme case ($p=\infty$), 
we have $ \lim_{p\to\infty}\gamma(d,p,\alpha)=\{1-\alpha(d-1)/(d-2)\}/d$ and hence
\begin{equation*}
 \max\left(0,1-\sigma^2\frac{2\{(d-2)-\alpha(d-1)\}}{d \left\{\max |z_i|\right\}^{2-\alpha}|z_i|^{\alpha}}
\right)z_i
\end{equation*}
with $0\leq \alpha<(d-2)/(d-1)$ is minimax.

\begin{remark}\label{rem:decreasing}
The solution of $ \Psi_\phi(v)=0$ or $ g_\phi(v)=1/\lambda$ for any $\lambda>0$,
is 
\begin{equation*}
 \phi_{\DS}(v)=\frac{2(d-2)\gamma(d,p,\alpha)}
{1+\lambda v^{d-2-\alpha(d-1)}},
\end{equation*}
under which \cite{Dasgupta-Strawderman-1997} showed 
the risk of the estimator with $ \phi_{\DS}(v)$ 
is exactly equal to $d$ when $p=2$ and $\alpha=0$.
Actually it is 
related to the concept of ``near unbiasedness'' or ``approximate unbiasedness''
in the literature of SCAD (smoothly clipped absolute deviation) 
including \cite{Antoniadis-Fan-2001}. 
Since $ \phi_{\DS}(v) $ is monotone decreasing and approaches $0$ as $v\to\infty$,
unnecessary modeling biases are effectively avoided with $ \phi_{\DS}(v) $.
\end{remark}

\section{Minimaxity with sparsity: unknown scale}
\label{sec:main_2}
In this section, we assume that $\sigma^2$ is unknown
and that $S\sim \sigma^2\chi_n^2$ is additionally observed.
We establish minimaxity result of the shrinkage estimators
$\hat{\theta}_\phi$ with the $i$-th component given by
\begin{equation}\label{eq:shrinkage_estimator_unknown_1}
\begin{split}
 \hat{\theta}_{i\phi} 
 &= \left(1-\frac{\hat{\sigma}^2\phi(\|z\|_p/\sqrt{\hat{\sigma}^2})}{\|z\|_p^{2-\alpha}|z_i|^\alpha}\right)z_{i} \\
 &= \left(1-\frac{s}{n+2}\frac{\phi\left(\sqrt{n+2}\|z\|_p/\sqrt{s}\right)}{\|z\|_p^{2-\alpha}|z_i|^\alpha}\right)z_{i}
\end{split} 
\end{equation}
where $\hat{\sigma}^2=s/(n+2)$.

\begin{lemma}\label{lem:stein_unknown}
Assume that $\phi(u)$ is, non-negative, bounded and absolutely continuous and that $ 0\leq \alpha < 1$. 
Then the risk function of the estimator $\hat{\theta}_\phi$ is
\begin{align} \label{eq:risk_unknown}
 E\left[ \frac{\|\hat{\theta}_\phi-\theta\|_2^2}{\sigma^2}\right]  
 \leq d+ E\left[\sum\nolimits_{i}\left\{\frac{|z_i|}{\|z\|_p}\right\}^{p-\alpha}
\!\! \frac{\phi(u)}{u^2}\left(\Psi_\phi(u)-\frac{2u\phi'(u)}{n+2} \right)\right]
\end{align}
where $u=\|z\|_p/\sqrt{\hat{\sigma}^2}$ and
$ \Psi_\phi(u)$ is given by \eqref{eq:Psi_phi}.
\end{lemma}
\begin{proof}
From the invariance with respect to the transformation,
$z\to cz$ and $s\to c^2 s$,
we can take $c=1/\sigma$ and hence, without the loss of generality, $\sigma^2=1$
is assumed in the proof.
Let $v=\|z\|_p$ and $u=v/\sqrt{\hat{\sigma}^2}$. 
Componentwisely we have
\begin{equation}\label{eq:quadratic_2_unknown}
\begin{split}
 (\hat{\theta}_{i}-\theta_{i})^2 
&=\left\{\left(1-\frac{\phi(u)\hat{\sigma}^2}{v^{2-\alpha}|z_i|^{\alpha}}\right)z_{i}-\theta_{i}\right\}^2 \\
&=(z_{i}-\theta_{i})^2+
\frac{\phi^2(u)\{\hat{\sigma}^2\}^2}{v^{2(2-\alpha)}}|z_{i}|^{2(1-\alpha)} -2\hat{\sigma}^2(z_{i}-\theta_{i})
 \frac{\phi(u)z_i}{v^{2-\alpha}|z_i|^{\alpha}}
\end{split} 
\end{equation}
 and hence
\begin{equation}\label{eq:quadratic_3_unknown}
\begin{split}
\sum_{i=1}^d (\hat{\theta}_{i}-\theta_{i})^2
&=\sum_{i=1}^d(z_{i}-\theta_{i})^2+
 \frac{\phi^2(u)\{\hat{\sigma}^2\}^2}{v^{2(2-\alpha)}}\sum_{i=1}^d|z_{i}|^{2(1-\alpha)} \\
 &\quad -2\hat{\sigma}^2\sum_{i=1}^d(z_{i}-\theta_{i})\left\{ \phi(u)v^{\alpha-2}|z_i|^{-\alpha}z_i\right\}.
\end{split} 
\end{equation}
For the third term of the right-hand side of \eqref{eq:quadratic_3_unknown}, 
the Stein identity given by \eqref{Stein_Identity} is applicable.
By \eqref{eq:dif_p_norm},
the differentiation of $\phi(v/\sqrt{\hat{\sigma}^2})v^{\alpha-2}|z_i|^{-\alpha}z_i$ with respect to
$z_i$ is 
\begin{align*}
&\frac{(1-\alpha)\phi(v/\sqrt{\hat{\sigma}^2})}{v^{2-\alpha}}|z_i|^{-\alpha} 
 +\frac{(\alpha-2)\phi(v/\sqrt{\hat{\sigma}^2})}{v^{p+2-\alpha}}|z_i|^{p-\alpha}
 +\frac{\phi'(v/\sqrt{\hat{\sigma}^2})}{\sqrt{\hat{\sigma}^2}v^{p+1-\alpha}}|z_i|^{p-\alpha} \\
 &=\frac{\phi(u)}{v^{p+2-\alpha}}
 \left((1-\alpha)v^{p} |z_i|^{-\alpha} +\left\{(\alpha-2)+u\frac{\phi'(u)}{\phi(u)}\right\}|z_i|^{p-\alpha}\right).
\end{align*}
 By the inequality \eqref{eq:cor_ineq} 
and the Stein identity, we have
\begin{align*}
& -2E\left[ \hat{\sigma}^2\sum_{i=1}^d(z_{i}-\theta_{i})
\left\{ \phi(v/\sqrt{\hat{\sigma}^2})v^{\alpha-2}|z_i|^{-\alpha}z_i\right\}\right] \\
 & \leq - E\left[\frac{\hat{\sigma}^2\sum\nolimits_{i}|z_i|^{p-\alpha}}{v^{2-\alpha+p}}
\phi(u) \left(2(d-2)-\alpha(d-1)+2u\frac{\phi'(u)}{\phi(u)}\right)\right] \\
&=- E\left[\sum\nolimits_{i}\left\{\frac{|z_i|}{\|z\|_p}\right\}^{p-\alpha}\!\!
\frac{\phi(u)}{u^2} \left(2(d-2)-\alpha(d-1)+2u\frac{\phi'(u)}{\phi(u)}\right)\right].
\end{align*}
For the second term of the right-hand side of \eqref{eq:quadratic_3_unknown}, 
a well known identity for chi-square distributions
(see e.g.~\cite{Efron-Morris-1976})
\begin{align}\label{chi-square}
 E\left[sh(s)\right]=\sigma^2E\left[n h(s)+2sh'(s)\right]
\end{align}
for $s\sim \sigma^2\chi^2_n$ is applicable.
 The differentiation of 
 \begin{align*}
  \frac{\phi^2(v/\sqrt{\hat{\sigma}^2})\{\hat{\sigma}^2\}^2}{s}
  =\frac{\phi^2(\sqrt{n+2}v/\sqrt{s})s}{(n+2)^2},
 \end{align*}
 with respect to $s$, is
\begin{equation}\label{diff_s}
\begin{split}
 & \frac{\phi^2(\sqrt{n+2}v/\sqrt{s})-\phi(\sqrt{n+2}v/\sqrt{s})
 \phi'(\sqrt{n+2}v/\sqrt{s})\sqrt{n+2}v/s^{1/2}}{(n+2)^2} \\
 &=\frac{\phi^2(u)-u\phi(u)\phi'(u)}{(n+2)^2}.
\end{split}
\end{equation} 
Hence, by the identity \eqref{chi-square} with \eqref{diff_s}, we have
 \begin{align}\label{eq:chi_chi}
  E_{s|v}\left[\phi^2(u)\{\hat{\sigma}^2\}^2\right]=
  E_{s|v}\left[\hat{\sigma}^2\phi(u)\left\{\phi(u)-\frac{2}{n+2}u\phi'(u)\right\}\right].
 \end{align}
Further, by \eqref{eq:cor_ineq_111} and \eqref{eq:chi_chi}, we have
 \begin{align*}
& E\left[\frac{\phi^2(v/\sqrt{\hat{\sigma}^2})\{\hat{\sigma}^2\}^2}{v^{2(2-\alpha)}}\sum_{i=1}^d|z_{i}|^{2(1-\alpha)}\right] \\
& \leq \max(1,d^{(p+\alpha-2)/p})
 E\left[\frac{\sum\nolimits_{i}|z_i|^{p-\alpha}\hat{\sigma}^2}{v^{2-\alpha+p}}
  \phi(u)\left(\phi(u)-\frac{2}{n+2}u\phi'(u)\right)\right] \\
& = \max(1,d^{(p+\alpha-2)/p})
  E\left[
\sum\nolimits_{i}\left\{\frac{|z_i|}{\|z\|_p}\right\}^{p-\alpha}\!\!
\frac{\phi(u)}{u^2}\left(\phi(u)-\frac{2}{n+2}u\phi'(u)\right)\right]  .
 \end{align*}
\end{proof}
By Lemma \ref{lem:stein_unknown},
a sufficient condition for $E[\|\hat{\theta}-\theta\|_2^2]\leq d$ is 
\begin{equation}
\Psi_\phi(u)-\frac{2u\phi'(u)}{n+2}\leq 0
\end{equation}
as well as the assumptions of Lemma \ref{lem:stein_unknown}.
When $\phi$ is monotone non-decreasing, as in Theorem \ref{thm:minimax_0} for the known scale case,
we easily have a following result for minimaxity.
\begin{thm}
\label{thm:minimax_0_unknown}
Assume $d\geq 3$ and $0\leq \alpha <(d-2)/(d-1)$. 
Assume $\phi(u)$ is absolutely continuous, monotone non-decreasing and 
\begin{equation*}
 0\leq \phi(u)\leq 2(d-2)\gamma(d,p,\alpha)
\end{equation*}
where $\gamma(d,p,\alpha)$ is given by \eqref{eq:gamma}.
Under unknown $\sigma^2$, the shrinkage estimator
 $\hat{\theta}_\phi$, with the $i$-th component,
\begin{equation*}
 \hat{\theta}_{i\phi} 
  = \left(1-\frac{\hat{\sigma}^2\phi(\|z\|_p/\sqrt{\hat{\sigma}^2})}
     {\|z\|_p^{2-\alpha}|z_i|^\alpha}\right)z_{i}
\end{equation*}
is minimax.
\end{thm}
Hence Theorem \ref{thm:minimax_0_unknown} guarantees that Theorem \ref{thm:minimax_0} remains true
if $\sigma^2$ is replaced by the estimator $\hat{\sigma}^2=s/(n+2)$.
By following \cite{Efron-Morris-1976} and using the relation \eqref{eq:derivative},
a more general theorem corresponding to Theorem \ref{thm:minimax} is given as follows.
\begin{thm}
\label{thm:minimax_unknown}
Assume $d\geq 3$ and $0\leq \alpha <(d-2)/(d-1)$. 
Assume $\phi(u)$ is absolutely continuous and 
\begin{equation*}
 0\leq \phi(u)\leq 2(d-2)\gamma(d,p,\alpha)
\end{equation*}
where $\gamma(d,p,\alpha)$ is given by \eqref{eq:gamma}.
Further, for all $u$ with $ \phi(u)< 2(d-2)\gamma(d,p,\alpha)$
\begin{equation*}
g_\phi(u)= \frac{u^{d-2-\alpha(d-1)}\phi(u)}
{\{2(d-2)\gamma(d,p,\alpha)-\phi(u)\}^{1+2\{d-2-\alpha(d-1)\}/(n+2)}}
\end{equation*}
is assumed to be non-decreasing. 
Further if there exists $u_*>0$ 
such that $ \phi(u)= 2(d-2)\gamma(d,p,\alpha)$,
then $ \phi(u)$ is assumed equal to 
$2(d-2)\gamma(d,p,\alpha)$ for all $u\geq u_*$.
Then $\hat{\theta}_\phi$ is minimax.
\end{thm}
We see that Theorem \ref{thm:minimax} for known $\sigma^2$
guarantees minimaxity of  $\hat{\theta}_\phi$ with $\phi$ which is not monotone non-decreasing.
As I mentioned in Remark \ref{rem:decreasing}, even a monotone decreasing $  \phi_{\DS}(v)$,
which is the solution $g_\phi(u)=\lambda$, leads minimaxity.
In unknown variance case, however,
the solution of $g_\phi(u)=\lambda$ in Theorem \ref{thm:minimax_unknown}, is not tractable.
An alternative to $  \phi_{\DS}(v)$ is
\begin{equation*}
 \tilde{\phi}_{\DS}(u)=\frac{2(d-2)\gamma(d,p,\alpha)}{1+\lambda u^l},
\end{equation*}
 where
\begin{align*}
 l=\frac{d-2-\alpha(d-1)}{1+2\{d-2-\alpha(d-1)\}/(n+2)},
\end{align*}
By straightforward calculation, $g_\phi(u)$ with $ \tilde{\phi}_{\DS}(u)$ is increasing.

\appendix
\section{Some inequalities}
\label{sec:prel}
Here we summarize some inequalities which are used in the main article.
\begin{lemma}\label{lem:norm}
\begin{enumerate}
 \item \label{lem:norm:1}
Let $q > r > 0$. Then 
\begin{equation}\label{eq:pnorm_1}
 \|z\|^r_q \leq \|z\|^r_r \leq d^{1-r/q}\|z\|^r_q.
\end{equation}
\item \label{lem:norm:0}
Let $q\geq 0$ and $r\geq 0$.
Then
\begin{equation}\label{eq:lem:norm:0}
 d \sum\nolimits_{i=1}^d |z_i|^{q-r}
\leq
\sum\nolimits_{i=1}^d |z_i|^{-r}\sum\nolimits_{i=1}^d |z_i|^{q}.
\end{equation}
 \item \label{lem:norm:2}
       Let $a\geq 0$ and $ b\leq 1$. Assume $\sum\nolimits_{i=1}^d s_i=1$ and $s_i\geq 0$ for
       all $i$. Then
      \begin{equation}
       \frac{\sum\nolimits_{i=1}^ds_i^a}{\sum\nolimits_{i=1}^d s_i^b}\leq \max(1,d^{b-a}).
      \end{equation}
\end{enumerate}
\end{lemma}
\begin{proof}
\mbox{}[Part \ref{lem:norm:1}]
 In the first inequality, we have
\begin{equation*}
\frac{\|z\|^r_q}{\|z\|^r_r}
=\left\{\frac{\|z\|^q_q}{\|z\|^q_r}\right\}^{r/q}  
=\left(\sum_{i=1}^d \left\{\frac{|z_i|^r}{\|z\|^r_r}\right\}^{q/r}\right)^{r/q}  
\leq \left(\sum_{i=1}^d \frac{|z_i|^r}{\|z\|^r_r}\right)^{r/q}  =1
\end{equation*}
since $ |z_i|^r/\|z\|^r_r \leq 1$ and $q/r\geq 1$.
In the second inequality, 
let $X$ be a discrete random variable with
 the probability mass function $\PR(X=|z_1|^r)=\PR(X=|z_2|^r)=\dots=\PR(X=|z_d|^r)=1/d$.
Then
\begin{equation*}
 \|z\|_{r}^{r}/d=E[X] \leq \left\{E[X^{q/r}]\right\}^{r/q}
=\left\{\|z\|_{q}^{q}/d\right\}^{r/q}=d^{-r/q}\|z\|_{q}^{r}
\end{equation*}
where $q/r>1$ and the inequality is from Jensen's inequality.

[Part \ref{lem:norm:0}]
 Let $X$ be a discrete random variable with  the probability mass function
 $\PR(X=|z_1|)=\PR(X=|z_2|)=\dots=\PR(X=|z_d|)=1/d$. Then we have
\begin{equation*}
 \frac{\sum\nolimits_{i=1}^d |z_i|^{q-r}}{d}=E[X^{p-r}], \quad
 \frac{\sum\nolimits_{i=1}^d |z_i|^{q}}{d}=E[X^{q}], \quad \frac{\sum\nolimits_{i=1}^d |z_i|^{-r}}{d}=E[X^{-r}].
\end{equation*}
From the correlation inequality $E[X^{q-r}] \leq E[X^{q}]E[X^{-r}]$,
the inequality \eqref{eq:lem:norm:0} follows.

[Part \ref{lem:norm:2}] Let $f(\bm{s},c)=\sum\nolimits_{i=1}^d s_i^c$
 with $\bm{s}=(s_1,\dots,s_d)$.
 For any fixed $\bm{s}$, $f(\bm{s},c) $ is non-increasing in $c$.
For $a\geq b$, we have  
clearly $f(\bm{s},a)/f(\bm{s},b)\leq 1$ and the equality is attained by $\bm{s}=(1,0,\dots,0)$.
When $a<b$, we have  $0\leq a<b\leq 1$ from the assumption and hence
\begin{align*}
 d=f(\bm{s},0)\geq f(\bm{s},a) \geq f(\bm{s},b) \geq f(\bm{s},1)=1
\end{align*}
and $1\leq f(\bm{s},a)/f(\bm{s},b)\leq d$ for any $\bm{s}$. 
By the method of Lagrange multiplier, $\hat{\bm{s}}=(1,\dots,1)/d$ gives
the maximum value, $ f(\hat{\bm{s}},a)/f(\hat{\bm{s}},b)=d^{b-a}$.
\end{proof}


\begin{thebibliography}{9}

\bibitem[\protect\citeauthoryear{Antoniadis and
  Fan}{2001}]{Antoniadis-Fan-2001}
\begin{barticle}[author]
\bauthor{\bsnm{Antoniadis},~\bfnm{Anestis}\binits{A.}} \AND
  \bauthor{\bsnm{Fan},~\bfnm{Jianqing}\binits{J.}}
(\byear{2001}).
\btitle{Regularization of wavelet approximations}.
\bjournal{J. Amer. Statist. Assoc.}
\bvolume{96}
\bpages{939--967}.
\bnote{With discussion and a rejoinder by the authors}.
\bmrnumber{1946364}
\end{barticle}
\endbibitem

\bibitem[\protect\citeauthoryear{Baranchik}{1964}]{Baranchik-1964}
\begin{btechreport}[author]
\bauthor{\bsnm{Baranchik},~\bfnm{A.~J.}\binits{A.~J.}}
(\byear{1964}).
\btitle{Multiple regression and estimation of the mean of a multivariate normal
  distribution}
\btype{Technical Report} No. \bnumber{51},
\binstitution{Department of Statistics, Stanford University}.
\end{btechreport}
\endbibitem

\bibitem[\protect\citeauthoryear{Baranchik}{1970}]{Baranchik-1970}
\begin{barticle}[author]
\bauthor{\bsnm{Baranchik},~\bfnm{A.~J.}\binits{A.~J.}}
(\byear{1970}).
\btitle{A family of minimax estimators of the mean of a multivariate normal
  distribution}.
\bjournal{Ann. Math. Statist.}
\bvolume{41}
\bpages{642--645}.
\bmrnumber{0253461}
\end{barticle}
\endbibitem

\bibitem[\protect\citeauthoryear{Dasgupta and
  Strawderman}{1997}]{Dasgupta-Strawderman-1997}
\begin{barticle}[author]
\bauthor{\bsnm{Dasgupta},~\bfnm{Anirban}\binits{A.}} \AND
  \bauthor{\bsnm{Strawderman},~\bfnm{William~E.}\binits{W.~E.}}
(\byear{1997}).
\btitle{All estimates with a given risk, {R}iccati differential equations and a
  new proof of a theorem of {B}rown}.
\bjournal{Ann. Statist.}
\bvolume{25}
\bpages{1208--1221}.
\bmrnumber{1447748}
\end{barticle}
\endbibitem

\bibitem[\protect\citeauthoryear{Efron and Morris}{1976}]{Efron-Morris-1976}
\begin{barticle}[author]
\bauthor{\bsnm{Efron},~\bfnm{Bradley}\binits{B.}} \AND
  \bauthor{\bsnm{Morris},~\bfnm{Carl}\binits{C.}}
(\byear{1976}).
\btitle{Families of minimax estimators of the mean of a multivariate normal
  distribution}.
\bjournal{Ann. Statist.}
\bvolume{4}
\bpages{11--21}.
\bmrnumber{0403001}
\end{barticle}
\endbibitem

\bibitem[\protect\citeauthoryear{James and Stein}{1961}]{James-Stein-1961}
\begin{bincollection}[author]
\bauthor{\bsnm{James},~\bfnm{W.}\binits{W.}} \AND
  \bauthor{\bsnm{Stein},~\bfnm{Charles}\binits{C.}}
(\byear{1961}).
\btitle{Estimation with quadratic loss}.
In \bbooktitle{Proc. 4th {B}erkeley {S}ympos. {M}ath. {S}tatist. and {P}rob.,
  {V}ol. {I}}
\bpages{361--379}.
\bpublisher{Univ. California Press}, \baddress{Berkeley, Calif.}
\bmrnumber{0133191}
\end{bincollection}
\endbibitem

\bibitem[\protect\citeauthoryear{Stein}{1956}]{Stein-1956}
\begin{binproceedings}[author]
\bauthor{\bsnm{Stein},~\bfnm{Charles}\binits{C.}}
(\byear{1956}).
\btitle{Inadmissibility of the usual estimator for the mean of a multivariate
  normal distribution}.
In \bbooktitle{Proceedings of the {T}hird {B}erkeley {S}ymposium on
  {M}athematical {S}tatistics and {P}robability, 1954--1955, vol. {I}}
\bpages{197--206}.
\bpublisher{University of California Press}, \baddress{Berkeley and Los
  Angeles}.
\bmrnumber{0084922}
\end{binproceedings}
\endbibitem

\bibitem[\protect\citeauthoryear{Stein}{1981}]{Stein-1981}
\begin{barticle}[author]
\bauthor{\bsnm{Stein},~\bfnm{Charles~M.}\binits{C.~M.}}
(\byear{1981}).
\btitle{Estimation of the mean of a multivariate normal distribution}.
\bjournal{Ann. Statist.}
\bvolume{9}
\bpages{1135--1151}.
\bmrnumber{630098}
\end{barticle}
\endbibitem

\bibitem[\protect\citeauthoryear{Zhou and Hwang}{2005}]{Zhou-Hwang-2005}
\begin{barticle}[author]
\bauthor{\bsnm{Zhou},~\bfnm{Harrison~H.}\binits{H.~H.}} \AND
  \bauthor{\bsnm{Hwang},~\bfnm{J.~T.~Gene}\binits{J.~T.~G.}}
(\byear{2005}).
\btitle{Minimax estimation with thresholding and its application to wavelet
  analysis}.
\bjournal{Ann. Statist.}
\bvolume{33}
\bpages{101--125}.
\bmrnumber{2157797}
\end{barticle}
\endbibitem

\end{thebibliography}
\end{document}